\documentclass[11pt,reqno]{amsart}
\usepackage{amsmath, latexsym, amsfonts, amssymb,amsthm, amscd,epsfig,enumerate}
\usepackage{subfigure}
\advance\hoffset -.75cm

\usepackage{amsfonts}
\usepackage{latexsym}
\usepackage{amsmath}
\usepackage{amssymb}
\usepackage[usesnames]{xcolor}

\definecolor{refkey}{gray}{.75}
\definecolor{labelkey}{gray}{.75}

\newcommand{\N}{\mathbb N}

\renewcommand{\thesubfigure}{\arabic{subfigure}}
\makeatletter
\renewcommand{\@thesubfigure}{\tiny Figure \thesubfigure: \space}
\renewcommand{\p@subfigure}{}
\makeatother

\newtheorem{teo}{Theorem}[section]

\newtheorem{cor}[teo]{Corollary}

\newtheorem{pro}[teo]{Proposition}
\newtheorem{defn}[teo]{Definition}
\newtheorem{exmp}[teo]{Example}

\begin{document}

\baselineskip .6 cm

\title[Moving closer: contractive maps on discrete metric spaces and graphs]
{Moving closer: contractive maps on discrete metric spaces and graphs}


\author[F.~Zucca]{Fabio Zucca}
\address{F.~Zucca, Dipartimento di Matematica,
Politecnico di Milano,
Piazza Leonardo da Vinci 32, 20133 Milano, Italy.}
\email{fabio.zucca\@@polimi.it}

\begin{abstract}
We consider discrete metric spaces and we look for nonconstant contractions.
We introduce the notion of contractive map and we characterize the spaces
with nonconstant contractive maps. We provide some examples to discussion
the possible relations between contractions, contractive maps and constant
functions.
Finally we apply the main result to the
subgraphs of a nonoriented, connected graph.
\end{abstract}

\maketitle


\section{Introduction}
\label{sec:intro}
\setcounter{equation}{0}

Let us consider the following question: a group of people,
living in the same country, but in different places, want to
move in such a way that the distance among any two of them strictly
decreases. Unfortunately they cannot all fit into one single house. Hence, the constraints are:
\begin{enumerate}[(i)]
\item anyone can either
move from his place to another one's (no matter what the owner of the
``destination house''  does) or stay where he is,
\item they cannot all move
to the same place.
\end{enumerate}
Is it possible for them to move?

This question is a particular case of a more interesting
problem involving discrete, possibly infinite, metric spaces (we will come back to
the original problem 
in Section~\ref{sec:conclusions} at the end of the paper).
We say that a couple $(X,d)$ is a metric space if $X$ is a set and $d$ is 
real function, called \textit{distance}, defined on $X \times X$ such that (1) $d(x,y)=0$ if and only if $x=y$ (for all $x,y \in X$) and
(2) $d(x,y) \le d(x,z)+d(y,z)$ for all $x,y,z \in X$.
We do not require the distance to be finite; if $d(x,y)=\infty$ we imagine that $x$ and $y$
belong to two disjoint components of the space (see below a more detailed discussion on this,
 in the case of a graph).
From (2), with $y=x$, and (1) we have that
$0=d(x,x) \le 2d(x,z)$ (for all $x,z \in X$) hence $d$ is nonnegative.
Moreover, again from (2), using $z=x$, and (1),
we have that $d(x,y) \le d(x,x)+d(y,x) =d(y,x)$ (for all $x,y \in X$), thus, by symmetry,
we have $d(x,y)=d(y,x)$.
For some basic properties of metric spaces see for instance \cite[Chapter 2]{cf:Rudin}.

An example of a discrete metric space is given by a (nonoriented) graph
(see for instance \cite{cf:Bollob}). Roughly speaking, a graph is a (finite or infinite)
collection of points, called \textit{vertices}, along with a set of pairs of vertices,
called \textit{edges}. We say that it is possible to move (in one step) from a vertex
$x$ to a vertex $y$ if and only if $(x,y)$ is an edge. A \textit{path} is a concatenation of edges,
and the minimum number of steps required to go from $x$ to $y$, say $d(x,y)$,
is called \textit{natural distance} from $x$ to $y$.
To be precise, in order to have an actual distance we have to assume that if we can move in one step
from $x$ to $y$ then we can also go back in one step, that is, if $(x,y)$ is an edge then $(y,x)$ is
an edge (in this case the graph is called \textit{nonoriented}). 
If we can go, in a finite number of steps, from $x$ to $y$ and back then we say that $x$ and $y$ \textit{communicate}.
The set of vertices which communicate with $x$
is called the \textit{connected component} 
containing $x$; if there is only one connected component, we say that the graph is \textit{connected}.
Of course, if two vertices belong to two disjoint connected components then the distance between them
is infinite. In Figure~\ref{fig:1} there is an example of an infinite graph which is connected if the dashed
arrow is an edge and it has two connected components otherwise (the double arrows mean that you can move forth
and back between the points, thus the graph is nonoriented).

\begin{figure}
   \subfigure[\tiny The graph $\N \cup \N$.]{\includegraphics[height=3cm]{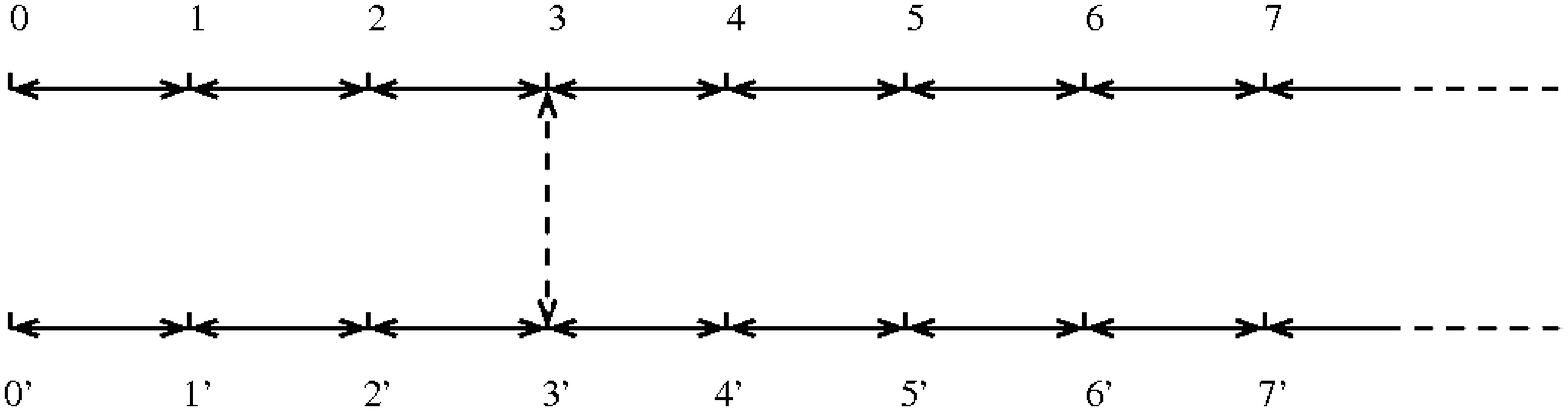}\label{fig:1}}
\end{figure}

The
central concept in this paper is the one of \textit{contractive map}
(see also \cite{cf:G1973, cf:handbook}).

\begin{defn}
Let $(X,d)$ and $(Y,\rho)$ be two metric spaces; a function $f:X \rightarrow Y$ is
\textsl{contractive} if and only if
for any $x,y \in X$, such that $x \not = y$,
we have $\rho(f(x),f(y)) < d(x,y)$. The map is called
a \textsl{contraction} if there exists $k <1$ such that
$\rho(f(x),f(y)) \le k d(x,y)$ for all $x,y \in X$.
\end{defn}

Roughly speaking a map is contractive if it (stricly) reduces the distances between points;
if the ratio of reduction has an upper bound which is (stricly) smaller than 1, then we have
a contraction.
Clearly, any contraction is a contractive map and the two
classes coincide
on finite metric spaces. We are (mainly) interested in the case $X=Y$ and $d=\rho$.
Observe that in our original problem, if we denote by $f(x)$ the new position
of the person who was at $x$ before,  what
we are looking for are nonconstant, contractive maps.
 Any constant map is trivially contractive;
when the converse is true?

The main result Theorem \ref{main} characterizes all the discrete
metric spaces, satisfying a certain property (see \eqref{ass:min} below)), which have
nonconstant contractive maps.
As a simply consequence of
Theorem~\ref{main} we will see that every contractive map on a
graph with the natural distance is a constant function if and only if the graph is connected.

\section{Main results and examples}
\label{sec:main}

Let us consider a discrete metric space $(X,d)$ (which is a
metric space such that for any given point $x\in X$ it is possible
to find $\varepsilon >0$ such that $d(x,y)\geq \varepsilon$ for all $y
\in X \setminus \{x\}$) and let us assume that
\begin{equation}\label{ass:min}
\exists x_0,y_0 \in X \hbox{ such that } 0< d(x_0,y_0) =
\min\{d(x,y): x,y \in X, x \not = y
\}
\end{equation}
(the existence of the minimum is implicitly assumed).
Roughly speaking, according to equation~\eqref{ass:min}, the distance either equals $0$ or it is
at least $d_0$; this implies that if $d(x,y)<d_0$ then $x=y$.
We define, once and for all, $d_0:=d(x_0,y_0)$ and we
introduce the equivalence relation $\sim$ defined by
$x \sim y$ if and only if there exists $\{x_i\}_{i=0}^n$ such that
$x_0=x$, $x_n=y$ and $d(x_i, x_{i+1})=d_0$ for any $i=0,1,\ldots, n-1$ (if $n>0$).
Hence, two points $x$ and $y$ are equivalent if and only if we can reach $y$ from
$x$ by performing a finite number of jumps of length $d_0$.
We denote by $[x]$, the equivalence class induced by $x \in X$,
that is, $[x]:=\{y:y \sim x\}$. As usual, the \textit{quotient space} $X/_\sim$ is the set of
the equivalence classes.
It is easy to show that, for any $x,y \in X$ such that $x \not = y$,
we have $y\in [x]$ if and only if there exists $z\in [x]$ satisfying
$d(y,z)=d_0$. We denote by $d_\sim$ the metric on $X/_\sim$ defined
by $d_\sim([x],[y]):= \inf_{x^\prime \in [x],y^\prime \in [y]} d(x,y)$.

Examples of discrete metric spaces satisfying equation~\eqref{ass:min}
are finite metric spaces and (nonoriented) graphs with their natural
distance; in the last case we have $d_0=1$ and the quotient space
$X/_\sim$ is the set
the connected components of the graph.

We are ready to state and prove the main result of this paper.

\begin{teo}\label{main}
Let $(X,d)$ be a metric space satisfying equation~\eqref{ass:min}; then TFAE
\begin{enumerate}[(i)]
\item  there exists a nonconstant contractive map on $X$;
\item there exists $x\in X$ such that $[x] \not = X$;
\item for every $x\in X$, we have $[x] \not = X$;
\item $\#X/_\sim>1$.
\end{enumerate}
Moreover there is a one-to-one map from the set of contractive maps $Cm(X)$ into
the set $Cm(X/_\sim;X)$ of contractive maps from ${X/_\sim}$ to $X$.
\end{teo}
\begin{proof}
$(i) \Longrightarrow (ii)$. If $f$ is contractive and
$d(x,y)=d_0$ then $d(f(x),f(y))< d_0$ which implies, by equation~\eqref{ass:min},
$f(x)=f(y)$; hence $f|_{[x]}$, the function $f$ restricted to the class $[x]$, is a constant
function for any $x\in X$. Hence, if $f$ is nonconstant we have $[x] \neq X$.

$(ii) \Longrightarrow (iii)$. Just remember that $\{[x]\}_{x\in X}$
is a partition of $X$; hence for all $y \in X$ we have either $[y]=[x]$ or
$[y] \subseteq [x]^c \subsetneq X$.

$(iii) \Longrightarrow (i)$. We have
that $Y:=X \setminus [x_0] \not = \emptyset$; let us define
a function $f$ by
\begin{equation*}
f(x):=
\begin{cases}
x_0 & \hbox{if } x\in [x_0] \\
y_0 & \hbox{if } x\in Y, \\
\end{cases}
\end{equation*}
where $d(x_0,y_0)=d_0$.
It is just a matter of easy computation
to show that this is a nonconstant contractive map
(indeed for any $x \in [x_0]$, $y \in Y$ we have
$d(x,y)>d_0$).

$(ii) \Longleftrightarrow (iv)$.
It is an easy consequence of the fact that for all $y,x \in X$, $y \not \in [x]$ if and only if $[y] \neq [x]$
which, in turn, is equivalent to $[y] \cap [x]=\emptyset$.
\smallskip

Finally, it is easy to show that, given a contractive map $f$, the map $\phi(f)$ is well defined by $\phi(f)([x]):=f(x)$
(since $f$ is constant
on every class $[x]$). It is straightforward to show that $\phi$ is injective
and, clearly, being $f$ a contractive map,  $d(\phi(f)([x]), \phi(f)([y]))=d(f(x),f(y))=d(f(x^\prime),f(y^\prime)) \le d(x^\prime, y^\prime)$
for all $x^\prime \sim x$ and $y^\prime \sim y$.
This implies 
that $d(\phi(f)([x]), \phi(f)([y])) \le d_\sim([x],[y])$, thus $\phi(f)$ is a contractive map
from ${X/_\sim}$ to $X$.
\end{proof}

Let us observe that if the range of the distance ${\rm Ran} (d)$ is a finite
set (take for
instance $X$ finite) then any contractive map $f$ is
actually a contraction and there is $n_0 \in \N$ such that the $n$-th iteration $f^{(n)}$
is  a constant map for all $n \ge n_0$ and $f^{(n)}=f^{(n_0)}$.
Indeed, $d(f^{(n)}(x),f^{(n)}(y)) \le k^n d(x,y) \le k^n \max_{x^\prime ,y^\prime  \in X} d(x^\prime,y^\prime)$,
hence if $n_0 >\log(d_0/\max_{x^\prime ,y^\prime  \in X} d(x^\prime,y^\prime))/\log(k)$ then $f^{(n_0)}$
is constant (say, $f^{(n_0)}=x_\infty$ for all $x \in X$).
If $n >n_0$ then $f^{(n)}(x)=f^{(n_0)}(f^{(n-n_0)}(x))=x_\infty$. Clearly $f(x_\infty)=
f(f^{(n_0)}(x_\infty))=f^{(n_0)}(f(x_\infty))=x_\infty$,
that is, $x_\infty$ is (the unique) fixed point of $f$.

Indeed, it is possible to prove (see \cite[Theorem 9.3]{cf:Rudin}) that for any contraction $f$ in a (complete) metric space there exists
$x_\infty=\lim_n f^{(n)}(x)$ exists
and it is the unique fixed point for $f$, that is, $f(x_\infty)=x_\infty$.
A generic contractive map $f$ has at most one fixed point,
since $x=f(x)$, $y=f(y)$ and $x \neq y$ implies $0 \neq d(x,y) =d(f(x),f(y)) < d(x,y)$ which
is a contradiction.
Nevertheless, if ${\rm Ran} (d)$ is not a finite
set then the set of fixed points of $f$ might be empty as Example~\ref{ex:empty} shows.
In the language of our original problem, $x$ is a fixed point if and only if the member of
the group living at $x$ does not move.

Moreover if $\#X/_\sim =1$ then the three classes of
contractive maps $Cm(X)$, contractions
$Ct(X)$ and constant functions $Cnst(X)$ coincide.

\begin{cor}\label{cor:main}
Let $(X,d)$ be a metric space such that there are just
two couples satisfying equation~\eqref{ass:min}
(namely $(x_0,y_0)$ and $(y_0,x_0)$); if $\# X>2$ then
there exists a non constant contractive map.
\end{cor}

\begin{exmp}\label{ex:empty}
 Consider $X=\N$ and define the distance as follows
\[
d(x,y):=
 \begin{cases}
  1 & (x,y) \in \{(0,1),(1,0)\} \\
 2  & (x,y) \in \{(0,2),(2,0),(1,2),(2,1)\} \\
 1+1/i & (x,y) \in \{(i,i+1),(i+1,i)\}\\
2+(1+1/2)+\cdots+(1+1/i) & (x,y) \in \{(0,i+1),(1,i+1),(i+1,0),(i+1,0)\} \\
(1+1/i)+\cdots+(1+1/j) & (x,y) \in \{(i,j+1),(j+1,i)\} \\
 \end{cases}
\]
where $2 \le i \le j$ (see Figure~\ref{fig:2} for a picture of a finite
portion of this space along with the distances between ``consecutive'' points).
In this case it is easy to prove that the following is a contracting
map without fixed points
\[
 f(x):=
\begin{cases}
 2 & x \in \{0,1\} \\
i+1 & x=i \ge 2. \\
\end{cases}
\]

\end{exmp}
\begin{figure}
   \subfigure[\tiny The metric space of Example~\ref{ex:empty}.]{\includegraphics[height=3cm]{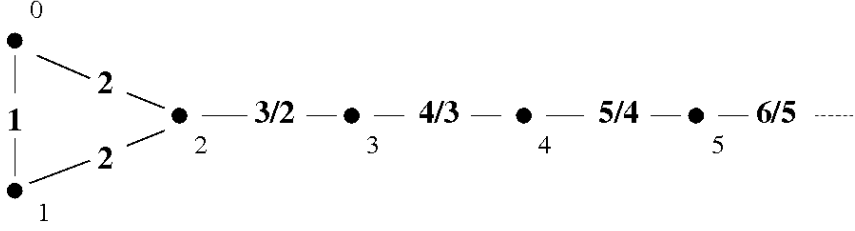}\label{fig:2}}
\end{figure}

In the following we compare the relations between
$Cm(X)$, $Ct(X)$ and $Cnst(X)$.
Clearly $Cm(X) \supseteq Ct(X) \supseteq Cnst(X)$; we provide examples
to show that, even if equation~\eqref{ass:min} holds,
all cases are possible. Given $A \subseteq X$, by
$\chi_A$ we mean the usual characteristic function
of $A$ (which equals $1$ on $A$ and $0$ elsewhere).
The reader is encouraged to verify that, in the following examples,
the spaces we define are indeed metric spaces.

\begin{exmp}\label{ex1}
Let $X:=\{0,1,2\}$ and $d$ be defined by $d(0,1)=d(1,0):=1$,
$d(1,2)=d(2,1):=2$,
$d(0,2)=d(2,0):=3$ and $0$ otherwise; then $f(x):=\chi_{\{2\}}(x)$ is a
nonconstant contraction.\hfill\break
Another example (in the infinite case) is the following.
Let $X=\N$ and define the distance by
\begin{equation*}
d(x,y):=
\begin{cases}
0 & \hbox{if } x=y, \\
1 & \hbox{if } (x,y) \in \{(0,1), (1,0)\}, \\
2 & \hbox{if } (x,y) \not \in \{(0,1), (1,0)\},\, x\not = y, \\
\end{cases}
\end{equation*}
then
$f(x):= \chi_{\{0,1\}^c}(x)$ is a nonconstant contraction.
In the previous examples $Cm(X)=Ct(X) \not = Cnst(X)$ since the range of the distance
is finite.
\end{exmp}

\begin{exmp}\label{ex2}
Let $X:=\N$ and let $d$ be defined as follows
\begin{equation*}
d(x,y):=
\begin{cases}
0 & \hbox{if } x=y, \\
1 & \hbox{if } (x,y) \in \{(0,1), (1,0)\}, \\
1+\frac{1}{y} & \hbox{if } 1\leq x < y \hbox{ or } x=0,\, y>1, \\
1+\frac{1}{x} & \hbox{if } 1\leq y < x \hbox{ or } y=0,\, x>1. \\
\end{cases}
\end{equation*}
In this case it is easy to check that this is a distance and that
\begin{equation*}
f(x):=
\begin{cases}
0 & \hbox{if } x\in \{0,1\}, \\
x+1 & \hbox{if } x>1, \\
\end{cases}
\end{equation*}
is a contractive map which is not a contraction while
\begin{equation*}
f(x):=
\begin{cases}
0 & \hbox{if } x \not = 2, \\
1 & \hbox{if } x=2, \\
\end{cases}
\end{equation*}
is a (nonconstant) contraction with 
$k$ equal to $2/3$.
Hence $Cm(X) \not = Ct(X)
\not = Cnst(X)$.
\end{exmp}

In Theorem~\ref{main}, to prove that every contractive map is a constant map, we require
that all the points belong to the same class, that is, we can reach $y$ from $x$ with
a finite number of consecutive steps of length $d_0$ (for any choice of $x,y \in X$,
$x \neq y$). To prove an analogous result for contractions we do not need such a
strong property: we simply require that any two different points can be joined by
a sequence of consecutive steps whose lengths are arbitrarily close to $d_0$.
The following result is used in Example~\ref{ex3} to construct a space where
every contraction is a constant map but there are nonconstant contractive maps.

\begin{pro}\label{pro:contrconst}
If 
equation \eqref{ass:min} holds and for all $x, x^\prime \in X$, $\varepsilon >1$
there exists $\{x_i\}_{i=1}^n \in X$ such that, $x_0=x$, $x_n=x^\prime$ and $d(x_i,x_{i+1}) \le \varepsilon d_0$ for all $i=0,1,\ldots
n-1$
then any contraction from $X$ into itself is constant.
\end{pro}

\begin{proof}
Suppose that $d(f(x),f(y)) \le k d(x,y)$ for all $x,y \in X$ where $k<1$. Define $\varepsilon:= 2/(k+1)$;
it is clear that for all $i=0,1,\ldots, n-1$ we have that 
\[
d(f(x_i),f(x_{i+1})) \le k d_0 \varepsilon =\frac{2k}{k+1} d_0 < d_0
\]
which implies $d(f(x_i),f(x_{i+1}))=0$ and $f(x_i)=f(x_{i+1})$.
Thus $f(x)=f(x^\prime)$.
\end{proof}

\begin{exmp}\label{ex3}
Let $X:=\N$ and let $d$ be defined as follows
\begin{equation*}
d(x,y):=
\begin{cases}
0 & \hbox{if } x=y, \\
1 & \hbox{if } (x,y) \in \{(0,1), (1,0)\}, \\
1+\frac{1}{x(y-x)} & \hbox{if } 1\leq x < y, \\
1+\frac{1}{y(x-y)} & \hbox{if } 1\leq y < x, \\
1+\frac{1}{y} & \hbox{if } x=0,\, y>1, \\
1+\frac{1}{x} & \hbox{if } y=0,\, x>1. \\

\end{cases}
\end{equation*}
It is easy to check that this is a distance and that
the hypotheses of Proposition~\ref{pro:contrconst} are satisfied (since $\lim_{y \to \infty}d(x,y)=1=d_0$, we can always
take $n=2$ and $x_1 \in \N$ sufficiently large).
Hence in this
case any contraction is a constant function.
Nevertheless
\begin{equation*}
f(x):=
\begin{cases}
0 & \hbox{if } x\in \{0,1\}, \\
x+1 & \hbox{if } x>1, \\
\end{cases}
\end{equation*}
is a (nonconstant) contractive map. Hence
$Cm(X) \not = Ct(X) = Cnst(X)$.
\end{exmp}
Let us consider now a nonoriented graph 
along with
its natural distance
$d$. 
If $A \subseteq X$, $d_A$ is the natural distance restricted to $A \times A$
and $d^A_0:= \min \{d_A(x,y):x,y\in A, x\not = y\}$, then
condition (ii) of Theorem~\ref{main} is equivalent to the existence,
given any couple of vertices $x,y \in A$, of a finite sequence
$\{x=:x_0, \ldots, x_n:=y\}$ of vertices in $A$
such that $d_A(x_i,x_{i+1})=d_0$
for all $i=0,\ldots, n-1$. In particular if $d^A_0:=1$ then any
contractive map on $A$ is constant if and only if $A$ is
a connected subgraph. By taking $A=X$ we have that there are no
nonconstant, contractive maps on $X$ if and only if $X$ is a connected
graph.

\section{Conclusions}
\label{sec:conclusions}

We come back to our original question: may the group of people move
according to the rules (i) and (ii) as stated in Section~\ref{sec:intro}?
We may reasonably suppose that the group is finite and with cardinality
strictly greater than $2$ (if there are just $2$ people then any contractive map
is constant and they cannot move). Hence the metric space is finite
and equation~\eqref{ass:min} is fulfilled. Moreover one can assume
that to different couples of places correspond different distances
(i.e.~if $(x,y)\not = (x_1,y_1)$ and $(x,y)\not = (y_1,x_1)$
then $d(x,y) \not = d(x_1,y_1)$ unless $x=y$ and $x_1=y_1$). Under
these assumptions, according to Corollary~\ref{cor:main}, we
know that the group can move. Moreover, since
the contractive map in this case is a contraction, there exists
one (and only one) person who does not move at all (this is the fixed point of
the contraction).

\end{document}